\documentclass[12pt]{article}
\usepackage[latin1]{inputenc}
\usepackage{lmodern}

\usepackage{amssymb, amsmath, amsthm}
\usepackage[a4paper,top=25mm,bottom=25mm,left=25mm,right=25mm]{geometry}
\usepackage{etex}

\usepackage{authblk} % for headings
\usepackage{pifont}
\usepackage{graphicx}
\usepackage[usenames,dvipsnames,svgnames,table]{xcolor}
\usepackage[figuresright]{rotating}
\usepackage{xtab} % tackle the long tables
\usepackage{longtable} % tackle the long tables
\usepackage{footnote}
\usepackage[stable]{footmisc}
\usepackage{chngpage} % allows for temporary adjustment of side margins
\usepackage{pdflscape} % landscape environment

\usepackage{pgfplots}
\usepackage{setspace}

\makesavenoteenv{tabular}
\usepackage{tabularx}
\usepackage{booktabs}
\usepackage{multirow}
\usepackage{threeparttable}
\usepackage[referable]{threeparttablex} % footnotes in tabu
\newcolumntype{R}{>{\raggedleft\arraybackslash}X}
\newcolumntype{L}{>{\raggedright\arraybackslash}X}
\newcolumntype{C}{>{\centering\arraybackslash}X}
\newcolumntype{A}{>{\columncolor{gray!25}}C}
\newcolumntype{a}{>{\columncolor{gray!25}}c}

\usepackage{dcolumn} % alignment to decimal points
\newcolumntype{.}{D{.}{.}{-1}}

\usepackage{tikz}
\usetikzlibrary{arrows}
\usepackage[semicolon,numbers]{natbib}
\usepackage[unicode]{hyperref} % [hidelinks]
\PassOptionsToPackage{unicode}{hyperref}
\usepackage{bookmark}
\usepackage{microtype}
\usepackage[justification=centerfirst]{caption}

\usepackage[labelformat=simple]{subcaption}

\DeclareCaptionLabelFormat{parenthesis}{(#2)}
\makeatletter
\renewcommand\p@subfigure{\arabic{figure}.}
\makeatother

\DeclareCaptionLabelFormat{parenthesis}{(#2)}
\makeatletter
\renewcommand\p@subtable{\arabic{table}.}
\makeatother

\usepackage{enumitem}

\setlist[itemize]{leftmargin=3\parindent}
\setlist[enumerate]{leftmargin=2\parindent}

\theoremstyle{plain}

\newtheorem{corollary}{Corollary}

\newtheorem{proposition}{Proposition}
\newtheorem{theorem}{Theorem}

\theoremstyle{definition}

\newtheorem{definition}{Definition}
\newtheorem{example}{Example}

\theoremstyle{remark}
\newtheorem{notation}{Notation}
\newtheorem{remark}{Remark}

\def\keywords{\vspace{.5em} % Add keywords
{\textit{Keywords}:\,\relax%
}}

\def\JEL{\vspace{.5em} % Add keywords
{\textbf{JEL classification number}:\,\relax%
}}

\def\AMS{\vspace{.5em} % Add keywords
{\textbf{AMS classification number}:\,}}

\author{L\'aszl\'o Csat\'o\thanks{~Department of Operations Research and Actuarial Sciences, Corvinus University of Budapest (BCE) and MTA-BCE ''Lend\"ulet'' Strategic Interactions Research Group, Budapest, Hungary \\ e-mail: laszlo.csato@uni-corvinus.hu} --
Lajos R\'onyai\thanks{~Informatics Laboratory, Institute for Computer Science and Control, Hungarian Academy of Sciences (MTA SZTAKI) and Budapest University of Technology and Economics (BME), Budapest, Hungary \\ e-mail: ronyai.lajos@sztaki.mta.hu}}
\title{Incomplete Pairwise Comparison Matrices and Weighting Methods\thanks{~We are grateful to S\'andor Boz\'oki for reading the manuscript and for useful advices. \newline
The research was supported by OTKA grants K 111797 and NK 105645.}}
\date{\today}

\begin{document}

\maketitle

\begin{abstract}
A special class of preferences, given by a directed acyclic graph, is considered.
%In this case, at least one natural order of the alternatives exists.
They are represented by incomplete pairwise comparison matrices as only partial information is available: for some pairs no comparison is given in the graph. A weighting method satisfies the linear order preservation property if it always results in a ranking such that an alternative directly preferred to another does not have a lower rank.
We study whether two procedures, the Eigenvector Method and the Logarithmic Least Squares Method meet this axiom. Both weighting methods break linear order preservation, moreover, the ranking according to the Eigenvector Method depends on the incomplete pairwise comparison representation chosen.


\JEL{C44}

\AMS{15A06, 90B50, 91B08}
\end{abstract}

\section{Introduction} \label{Theory}

Pairwise comparisons are widely used in multi-attribute decision making since Saaty published the AHP method \citep{Saaty1980}.
It is assumed that decision makers give a numerical answer to the question 'How many times is the $i$th alternative more important/better/favorable than the $j$th?', which are incorporated into a matrix with an appropriate size.

Let $\mathbb{R}^n_+$ denote the positive orthant of the $n$-dimensional Euclidean space and $\mathbb{R}^{n \times n}_+$ denote the set of positive matrices of size $n \times n$.

\begin{definition}
\emph{Pairwise comparison matrix}: Matrix $\mathbf{A}  = \left[ a_{ij} \right] \in \mathbb{R}^{n \times n}_+$ is a \emph{pairwise comparison matrix} if $a_{ji} = 1/a_{ij}$ for all $i,j = 1,2, \dots ,n$.
\end{definition}

The final aim of the use of pairwise comparisons is to determine a weight vector $\mathbf{w} = \left[ w_i \right] \in \mathbb{R}^n_+$ for the alternatives such that $w_i/w_j$ somehow approximates $a_{ij}$.

\begin{definition}
\emph{Consistency}: Pairwise comparison matrix $\mathbf{A}  = \left[ a_{ij} \right]$ is \emph{consistent} if $a_{ik} = a_{ij}a_{jk}$ for all $i,j,k = 1,2, \dots ,n$.
\end{definition}

Every consistent pairwise comparison matrix can be associated to a weight vector $\mathbf{w}$ where $a_{ij} = w_i/w_j$ for all $i,j = 1,2, \dots ,n$.
Vector $\mathbf{w}$ is unique up to multiplication by positive scalars.

Pairwise comparison matrices provided by decision makers are usually do not meet the consistency condition. In other words, they are \emph{inconsistent}. Then the real weight vector $\mathbf{w}$ can only be estimated on the basis of the inconsistent pairwise comparison matrix.
A number of weighting methods is proposed for this purpose.

\citet{Saaty1980} used the Perron theorem \citep{Perron1907}: a positive matrix has a dominant eigenvalue with multiplicity one and an associated strictly positive (right) eigenvector.

\begin{definition}
\emph{Eigenvector Method} ($EM$) \citep{Saaty1980}: $EM$ gives the weight vector $\mathbf{w}^{EM}(\mathbf{A}) \in \mathbb{R}^n_+$ for any pairwise comparison matrix $\mathbf{A}$ such that
\[
\mathbf{A} \mathbf{w}^{EM}(\mathbf{A}) = \lambda_{\max} \mathbf{w}^{EM}(\mathbf{A}),
\]
where $\lambda_{\max}$ denotes the maximal eigenvalue, also known as Perron eigenvalue, of matrix $\mathbf{A}$.
\end{definition}

Distance-minimization techniques minimize the function $\sum_i \sum_j d(a_{ij}, w_i/w_j)$, where $d(a_{ij}, w_i/w_j)$ is some sort of a distance of $a_{ij}$ from its approximation $w_i/w_j$. The following is an important example with $d(a_{ij}, w_i/w_j) = \left[ \log a_{ij} - \log \left( w_i / w_j \right) \right]^2$.

\begin{definition}
\emph{Logarithmic Least Squares Method} ($LLSM$) \citep{CrawfordWilliams1980, CrawfordWilliams1985, DeGraan1980}: $LLSM$ gives the weight vector $\mathbf{w}^{LLSM} (\mathbf{A}) \in \mathbb{R}^n_+$ for any pairwise comparison matrix $\mathbf{A}$ as the optimal solution of the problem:
\[
\min_{\mathbf{w} \in \mathbb{R}^n_+, \, \sum_{i=1}^n w_i = 1} \sum_{i=1}^n \sum_{j=1}^n \left[ \log a_{ij} - \log \left( \frac{w_i}{w_j} \right) \right]^2.
\]
\end{definition}

It may also happen that some pairwise comparisons are unknown due to the lack of available data, uncertain evaluations, or other problems. Incomplete pairwise comparison matrices were introduced in \citet{Harker1987}.

\begin{definition}
\emph{Incomplete pairwise comparison matrix}: Matrix $\mathbf{A} = \left[ a_{ij} \right]$ of size $n \times n$ is an \emph{incomplete pairwise comparison matrix} if $a_{ii} = 1$ for all $i = 1,2, \dots ,n$, and for all $i \neq j$, $a_{ji} = 1/a_{ij} \in \mathbb{R}_+$ or both $a_{ij}$ and $a_{ji}$ are missing.
\end{definition}

\begin{notation}
Missing elements of pairwise comparison matrices are denoted by $\ast$.
\end{notation}

\begin{example}
The following pairwise comparison matrix of size $4 \times 4$ is incomplete:
\[
 \mathbf{A}=
\left(
 \begin{array}{cccc}
 1		&   *       	&   a_{13}  	&   a_{14}  \\
 * 		&   1       	&   a_{23}  	&   *		\\
1/a_{13}&   1/a_{23}	&   1       	&   a_{34}  \\
1/a_{14}&   *       	&   1/a_{34}	&   1		\\
\end{array}
\right).
\]
\end{example}

Generalization of $EM$ to incomplete pairwise comparison matrices requires some comment on measuring inconsistency.
Saaty \citep{Saaty1980} defined the $CR$ index as
\[
CR(\mathbf{A}) =
\frac{\left( \lambda_{\max}(\mathbf{A}) - n \right) / (n-1)}{\left( \overline{\lambda_{\max}^{n \times n}} - n \right) / (n-1)} =
\frac{\lambda_{\max}(\mathbf{A}) - n}{\overline{\lambda_{\max}^{n \times n}} - n},
\]
where $\overline{\lambda_{\max}^{n \times n}}$ denotes the average value of the maximal eigenvalue of randomly generated pairwise comparison matrices of size $n \times n$ such that each element $a_{ij}$, $i < j$ is chosen from the set $\{  1/9; 1/8; \ldots; 1/2; 1; 2; \ldots; 8; 9 \}$ with equal probability. $CR(\mathbf{A})$ is a positive linear transformation of $\lambda_{\max}(\mathbf{A})$. $CR(\mathbf{A}) \geq 0$ and $CR(\mathbf{A}) = 0$ if and only if $\mathbf{A}$ is consistent. Saaty recommended the rule of acceptability $CR < 0.1$.

The idea that larger $\lambda_{\max}$ indicates higher ($CR$) inconsistency led \citep{ShiraishiObataDaigo1998,ShiraishiObata2002} to introduce variables for missing elements, arranged in vector $\mathbf{x}$ and consider the eigenvalue optimization problem
\begin{equation*}
\underset{{\mathbf{x} > 0}}{\min} \, \lambda_{max}(\mathbf{A}(\mathbf{x}))
\end{equation*}
in order to find a completion that minimizes the maximal eigenvalue, or, equivalently, $CR$.
%As in case of incomplete $LLSM$, uniqueness is closely related to the connectedness of $G$.

Extension of distance-based weighting methods to the incomplete case seems to be straightforward: when calculating the optimal weights, only the known terms are considered in the objective function \citep{Kwiesielewicz1996, BozokiFulopRonyai2010}.

\citet{BozokiFulopRonyai2010} discuss the question of uniqueness of the optimal solution for $EM$ and $LLSM$ in the incomplete case, solve the $LLSM$ problem\footnote{~See also \citet{KaiserSerlin1978}.} and propose an algorithm for finding the best completion of an incomplete pairwise comparison matrix according to $EM$. We will use their results extensively.

%The question of uniqueness of the optimal solution for $EM$ and $LLSM$ in the incomplete case occurs naturally. We will

This paper investigates a special class of preferences described by incomplete pairwise comparison matrices (Section~\ref{Sec2}), for which some natural rankings of the alternatives exist. Section~\ref{Sec3} reveals that $LLSM$ does not result in one of these orders. Section~\ref{Sec4} presents that $EM$ does not meet the required condition either. Moreover, the ranking depends on the representation chosen. These are the main results of our paper. Finally, in Section~\ref{Sec5}, we pose some related questions.

\section{Linear order preservation} \label{Sec2}

%Sometimes the decision maker is unable to give a numerical answer to the question 'How many times is the $i$th alternative more important/better/favorable than the $j$th?'.
Sometimes the decision maker can only provide an ordinal information such as the $i$th alternative is preferred to the $j$th \citep{Jensen1986, GenestLapointeDrury1993}. In this model, incomplete pairs (missing comparisons) are allowed but draws are excluded: when the $i$th and the $j$th alternatives have been compared, the $i$th or the $j$th is preferred to the other.

\begin{definition} %\label{Def5}
\emph{Ordinal pairwise comparison matrix}:
Incomplete pairwise comparison matrix $\mathbf{A} = \left[ a_{ij} \right]$ of size $n \times n$ is an \emph{ordinal pairwise comparison matrix} if $a_{ii} = 1$ for all $i = 1,2, \dots ,n$, and for all $i \neq j$, $a_{ij} \in \{ b;1/b \}$ or both $a_{ij}$ and $a_{ji}$ are missing. A real number $b>1$ is an arbitrarily fixed.
\end{definition}

Note that the value $b>1$ corresponds to the (strict) preference relation between the alternatives.

Ordinal pairwise comparison matrices can be represented by directed graphs. Let $\mathbf{A}$ be an ordinal pairwise comparison matrix of size $n \times n$. Then $G := (V,E)$ where $V = \{ 1,2, \dots ,n \}$, the vertices correspond to the alternatives, and $E = \{ e(i,j): a_{ij} = b, \, i \neq j \}$, there is a directed edge from vertex $i$ to vertex $j$ if and only if the $i$th alternative is preferred to the $j$th. The directed graph associated to an ordinal pairwise comparison matrix $\mathbf{A}$ is independent of the value $b>1$.

Note that different choice of the parameter $b>1$ is equivalent to taking a corresponding element-wise (positive) power of $\mathbf{A}$. In other words, the associated directed graph is the same for every $\mathbf{A}^{(h)} = \left[ a_{ij}^h \right]$, $h>0$.

\begin{definition} %\label{Def6}
\emph{Weak connectedness}:
Let $\mathbf{A} = \left[ a_{ij} \right]$ be an ordinal pairwise comparison matrix of size $n \times n$.
The directed graph associated to $\mathbf{A}$ is \emph{weakly connected} if for all $k,\ell = 1,2, \dots ,n$, there exists a sequence of alternatives $k = m_0, m_1, \dots ,m_{t-1},m_t = \ell$ such that $a_{m_{s-1} m_s}$ is known for all $s=1,2 \dots ,t$.
\end{definition}

In an ordinal pairwise comparison matrix represented by a weakly connected directed graph, all alternatives are compared directly or indirectly (i.e. through other alternatives).

%, so invariant under positive power by each element.

\begin{definition} %\label{Def6}
\emph{Existence of a linear order of the alternatives}:
Let $\mathbf{A} = \left[ a_{ij} \right]$ be an ordinal pairwise comparison matrix of size $n \times n$. There \emph{exists a linear order of the alternatives} if there is a permutation $\sigma: \{ 1;2; \dots ;n \} \rightarrow \{ 1;2; \dots ;n \}$ on the set of alternatives such that $\mathbf{C} = \left[ c_{ij} \right]$ is the permuted ordinal pairwise comparison matrix given by $c_{ij} = a_{\sigma{(i)} \sigma{(j)}}$ for all $i,j = 1,2, \dots ,n$ and $c_{ij} = b$ if $i < j$ and $c_{ij}$ is known.
\end{definition}

Existence of a linear order of the alternatives means that the ordinal pairwise comparison matrix can be permuted such that every known value above the diagonal is $b>1$. Regarding the directed graph representation, it is equivalent to acyclicity.

The following condition concerns the weighting methods for ordinal pairwise comparison matrices. A similar requirement, called \emph{Condition of Order Preservation} (COP), was introduced by \citet{BanaeCostaVansnick2008}. However, it is defined on complete pairwise comparison matrices and takes into account the intensity of preferences.

\begin{definition} %\label{Def7}
\emph{Linear order preservation} ($LOP$):
Let $\mathbf{A} = \left[ a_{ij} \right]$ be an ordinal pairwise comparison matrix of size $n \times n$ such that there exists a linear order of the alternatives. It can be assumed without loss of generality that $a_{ij} = b$ if $i < j$ and $a_{ij}$ is known.
A weighting method associating a vector $\mathbf{w}(\mathbf{A}) \in \mathbb{R}^n_+$ to $\mathbf{A}$ satisfies \emph{linear order preservation} if $w_i(\mathbf{A}) \geq w_j(\mathbf{A})$ for all $i < j$ such that $a_{ij}$ is known ($a_{ij} = b$).
\end{definition}

In an ordinal pairwise comparison matrix exhibiting a linear order of the alternatives, there exist some 'natural rankings'. Linear order preservation requires that the ranking according to the weighting method examined always corresponds to one of them.

Note that a weighting method associating the same weight for each alternative meets the property $LOP$.

\section{Linear order preservation and the Logarithmic \\ Least Squares Method} \label{Sec3}

In this section it will be scrutinized whether $LLSM$ satisfies the property $LOP$.

\begin{notation}
$\mathbf{y}(\mathbf{A}) \in \mathbb{R}^n$ is given by $y_i(\mathbf{A}) = \log w^{LLSM}_i(\mathbf{A})$ for all $i = 1,2, \dots ,n$.
\end{notation}

\begin{proposition} \label{Prop1}
Let $\mathbf{A} = \left[ a_{ij} \right]$ be an ordinal pairwise comparison matrix of size $n \times n$.
Vector $\mathbf{w}^{LLSM}(\mathbf{A})$ is unique if and only if the directed graph associated to $\mathbf{A}$ is weakly connected.
Then the ranking of alternatives is independent of the value of $b>1$, that is, 
\[
w^{LLSM}_i (\mathbf{A}) \geq w^{LLSM}_j (\mathbf{A}) \Leftrightarrow w^{LLSM}_i (\mathbf{A}^{(h)}) \geq w^{LLSM}_j (\mathbf{A}^{(h)})
\]
for all $i,j = 1,2, \dots ,n$ and $h>0$, where $\mathbf{A}^{(h)} = \left[ a_{ij}^h \right]$.
\end{proposition}

\begin{proof}
The necessary and sufficient condition for uniqueness is given by \citet[Theorem 4]{BozokiFulopRonyai2010}.

$\mathbf{y}(\mathbf{A}) = \mathbf{D}(\mathbf{A}) \mathbf{r}(\mathbf{A})$ where $\mathbf{r}(\mathbf{A}) = \left[ \sum_{j: a_{ij} \text{is known}} \log a_{ij} \right] \in \mathbb{R}^n$, so $\mathbf{r}(\mathbf{A}^{(h)}) = h \mathbf{r}(\mathbf{A})$ and $\mathbf{D}(\mathbf{A})$ depends only on the positions of known comparisons but is not affected by their values \citep[Remark 3]{BozokiFulopRonyai2010}. Therefore $\mathbf{y}(\mathbf{A}^{(h)}) = h \mathbf{y}(\mathbf{A})$, which proves Proposition~\ref{Prop1}.
\end{proof}

Since linear order preservation is based on the directed acyclic graph representation of an ordinal pairwise comparison matrix, Proposition~\ref{Prop1} states that it is meaningful to question whether $LLSM$ satisfies $LOP$.

\begin{corollary} \label{Col1}
It does not depend on the choice of $b>1$ whether $LLSM$ satisfies $LOP$ or not. In other words, $LLSM$ gives the same ranking for every ordinal pairwise comparison matrix associated to a given directed acyclic graph.
\end{corollary}

%Now we are able to present an important result.
Intuition appears to suggest that LLSM satisfies LOP. The first of our main results contradicts this expectation.

\begin{theorem} \label{Theo1}
$LLSM$ may violate $LOP$.
\end{theorem}

\begin{proof}
It is provided by Example \ref{Examp2}.

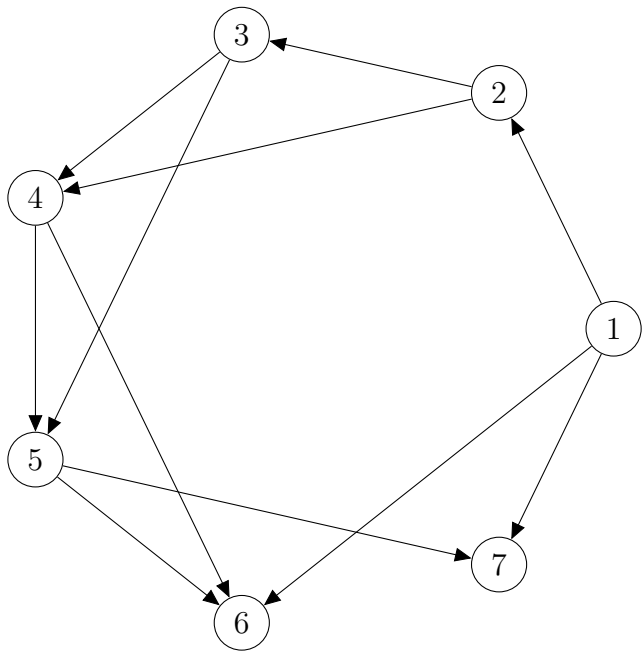
\begin{figure}[!ht]
\centering
\caption{The directed acyclic graph of Example~\ref{Examp2}}
\label{Fig1}
\begin{tikzpicture}[scale=1, auto=center, transform shape, >=triangle 45]
\tikzstyle{every node}=[draw,shape=circle];
  \node (n1)  at (0:4)  {$1$};
  \node (n2)  at (360/7:4)   {$2$};
  \node (n3)  at (2*360/7:4)   {$3$};
  \node (n4)  at (3*360/7:4)  {$4$};
  \node (n5)  at (4*360/7:4)  {$5$};
  \node (n6)  at (5*360/7:4)  {$6$};
  \node (n7)  at (6*360/7:4)  {$7$};

  \foreach \from/\to in {n1/n2,n1/n6,n1/n7,n2/n3,n2/n4,n3/n4,n3/n5,n4/n5,n4/n6,n5/n6,n5/n7}
    \draw [->] (\from) -- (\to);
\end{tikzpicture}
\end{figure}

\begin{example} \label{Examp2}
Consider the directed acyclic graph on Figure~\ref{Fig1}. The associated ordinal pairwise comparison matrix $\mathbf{A}$ is as follows:
\[
\mathbf{A}=
\left(
\begin{array}{cccccccc}
    1     & b     & \ast & \ast & \ast & b     & b \\
    1/b   & 1     & b     & b     & \ast & \ast & \ast \\
    \ast & 1/b   & 1     & b     & b     & \ast & \ast \\
    \ast & 1/b   & 1/b   & 1     & b     & b     & \ast \\
    \ast & \ast & 1/b   & 1/b   & 1     & b     & b \\
    1/b   & \ast & \ast & 1/b   & 1/b   & 1     & \ast \\
    1/b   & \ast & \ast & \ast & 1/b   & \ast & 1 \\
\end{array}
\right),
\]
where $b>1$.
\end{example}

In Example \ref{Examp2}, property $LOP$ is satisfied if $w_1(\mathbf{A}) \geq w_2(\mathbf{A})$, $w_1(\mathbf{A}) \geq w_6(\mathbf{A})$, $w_1(\mathbf{A}) \geq w_7(\mathbf{A})$ as well as $w_i(\mathbf{A}) \geq w_{i+1}(\mathbf{A})$ and $w_i(\mathbf{A}) \geq w_{i+2}(\mathbf{A})$ for all $i = 2,3,4,5$.

However, $LLSM$ results in
\[
\mathbf{y}(\mathbf{A}) = \left[
\begin{array}{ccccccc}
    34    & 36    & 24    & 1     & -14   & -42   & -39 \\
\end{array}
\right]^\top \log b / 49,
\]
namely, $w_1^{LLSM}(\mathbf{A}) < w_2^{LLSM}(\mathbf{A})$, in contradiction with preservation of linear order.
\end{proof}

\begin{remark} \label{Rem1}
Example~\ref{Examp2} is minimal regarding the number of alternatives ($7$) and among them, with respect to the number of known comparisons ($11$).\footnote{~It can be verified by brute force, examining all ordinal pairwise comparison matrices up to size $6 \times 6$. It is possible because of Corollary \ref{Col1}, which implies that comparisons above the diagonal may have essentially two 'values', known or missing. There exist $2^{15} = 32\,768$ acyclic directed graphs of size $6 \times 6$.}
However, there exist more than ten examples with $7$ alternatives, and some of them contain only $11$ known comparisons.
\end{remark}

\begin{remark} \label{Rem2}
There exist some examples to Theorem \ref{Theo1} with $8$ alternatives and $10$ known comparisons. Two of them are presented in Example~\ref{Examp3}.
\end{remark}

\begin{figure}[!ht]
\centering
\caption{The directed acyclic graphs of Example~\ref{Examp3}}
\label{Fig2}

\begin{subfigure}{0.48\textwidth}
	\centering
	\caption{Graph of matrix $\mathbf{A}$}
	\label{Fig2a}
	
\begin{tikzpicture}[scale=0.8, auto=center, transform shape, >=triangle 45]
\tikzstyle{every node}=[draw,shape=circle];
  \node (n1)  at (112.5:4)  {$1$};
  \node (n2)  at (67.5:4)   {$2$};
  \node (n3)  at (22.5:4)   {$3$};
  \node (n4)  at (337.5:4)  {$4$};
  \node (n5)  at (292.5:4)  {$5$};
  \node (n6)  at (247.5:4)  {$6$};
  \node (n7)  at (202.5:4)  {$7$};
  \node (n8)  at (157.5:4)  {$8$};

  \foreach \from/\to in {n1/n2,n1/n7,n1/n8,n2/n3,n2/n4,n3/n5,n4/n5,n5/n6,n5/n8,n6/n7}
    \draw [->] (\from) -- (\to);
\end{tikzpicture}
\end{subfigure}
\begin{subfigure}{0.48\textwidth}
	\centering
	\caption{Graph of matrix $\mathbf{A'}$}
	\label{Fig2b}
	
\begin{tikzpicture}[scale=0.8, auto=center, transform shape, >=triangle 45]
\tikzstyle{every node}=[draw,shape=circle];
  \node (n1)  at (112.5:4)  {$1$};
  \node (n2)  at (67.5:4)   {$2$};
  \node (n3)  at (22.5:4)   {$3$};
  \node (n4)  at (337.5:4)  {$4$};
  \node (n5)  at (292.5:4)  {$5$};
  \node (n6)  at (247.5:4)  {$6$};
  \node (n7)  at (202.5:4)  {$7$};
  \node (n8)  at (157.5:4)  {$8$};

  \foreach \from/\to in {n1/n4,n1/n8,n2/n3,n2/n8,n3/n4,n4/n5,n5/n6,n5/n7,n6/n8,n7/n8}
    \draw [->] (\from) -- (\to);
\end{tikzpicture}
\end{subfigure}
\end{figure}
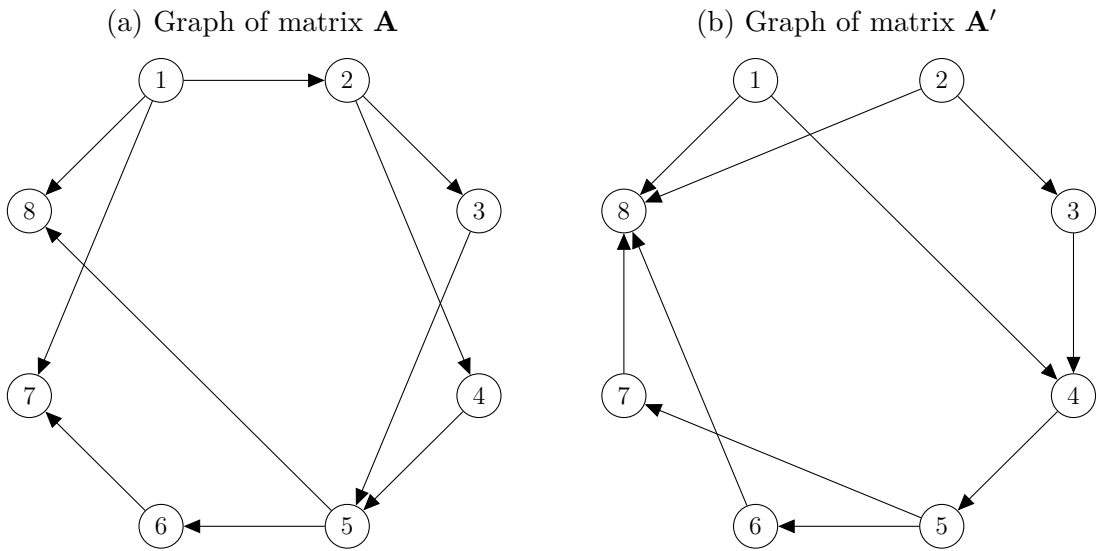

\begin{example} \label{Examp3}
Consider the directed acyclic graphs on Figure~\ref{Fig2} and the associated ordinal pairwise comparison matrices $\mathbf{A}$ and $\mathbf{A'}$. $LLSM$ gives
\[
\mathbf{y}(\mathbf{A}) = \left[
\begin{array}{cccccccc}
    95    & 103   & 43    & 43    & -17   & -65   & -113  & -89 \\
\end{array}
\right]^\top \log b / 128
\text{, and}
\]
\[
\mathbf{y}(\mathbf{A'}) = \left[
\begin{array}{cccccccc}
    71    & 95    & 47    & -1    & 7     & -53   & -53   & -113 \\
\end{array}
\right]^\top \log b / 128,
\]
where $w_1^{LLSM}(\mathbf{A}) < w_2^{LLSM}(\mathbf{A})$ and $w_4^{LLSM}(\mathbf{A'}) < w_5^{LLSM}(\mathbf{A'})$, in contradiction with preservation of linear order.
\end{example}

The violation of $LOP$ can be arbitrarily 'strong', regarding the difference of the weights of the alternatives involved.

\begin{theorem} \label{Theo2}
%Let $\mathbf{A} = (a_{ij})$ be an ordinal pairwise comparison matrix exhibiting a linear order of the alternatives with a given $b > 1$.
% Then $w_i^{LLSM}(\mathbf{A}) - w_j^{LLSM}(\mathbf{A})$ can be arbitrarily small for some $i < j$.
For every $K \in \mathbb{R}_+$ there exists an ordinal pairwise comparison matrix  $\mathbf{A} = \left[ a_{ij} \right]$ exhibiting a linear order of the alternatives with a given $b > 1$ such that $a_{ij}$ is known and $w_i^{LLSM}(\mathbf{A}) - w_j^{LLSM}(\mathbf{A}) \leq -K$ for some $i < j$.
\end{theorem}

\begin{proof}
It is provided by Example~\ref{Examp4} for any $k \geq 2$.

\begin{figure}[!ht]
\centering
\caption{The directed acyclic graph of Example~\ref{Examp4}}
\label{Fig3}
\begin{tikzpicture}[scale=1, auto=center, transform shape, >=triangle 45]
\tikzstyle{every node}=[draw,shape=circle];
  \node (n1)  at (112.5:4)  {$1$};
  \node (n2)  at (67.5:4)   {$2$};
  \node (n3)  at (22.5:4)   {$3$};
  \node (n4)  at (337.5:4)  {$4$};
  \node (n5)  at (292.5:4)  {$5$};
  \node (n6)  at (247.5:4)  {$6$};
  \node (n7)  at (202.5:4)  {$7$};
  \node (n8)  at (157.5:4)  {$8$};

  \foreach \from/\to in {n1/n2,n1/n7,n1/n8,n2/n3,n2/n4,n3/n5,n3/n6,n4/n5,n4/n6,n5/n7,n5/n8,n6/n7,n6/n8}
    \draw [->] (\from) -- (\to);
\end{tikzpicture}
\end{figure}
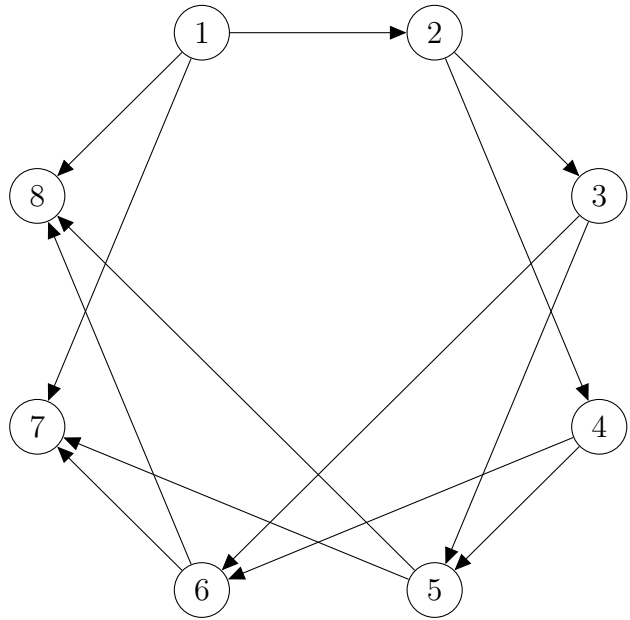

\begin{notation}
Let $S$ and $T$ be two sets of nodes in a directed acyclic graph.
$S \rightarrow T$ if and only if there exists an edge from every vertex $i \in S$ to every vertex $j \in T$. 
\end{notation}

\begin{example} \label{Examp4}
Consider the family of directed acyclic graphs with $n = k m + 2$ vertices where $C_i = \{ (i-1)k + 3, (i-1)k + 4, \dots , ik + 2 \}$ for all $i = 1,2, \dots m$ (so $C_i$ has $k$ elements) such that the edges are given by $\{ 1 \} \rightarrow \{ 2 \}$, $\{ 1 \} \rightarrow C_m$, $\{ 2 \} \rightarrow C_1$ and $C_i \rightarrow C_{i+1}$ for all $i = 1,2, \dots m-1$. Figure~\ref{Fig3} shows a member of this family if $k = 2$ and $m = 3$, that is, $C_1 = \{ 3,4 \}$, $C_1 = \{ 5,6 \}$ and $C_3 = \{ 7,8 \}$.
\end{example}

The directed graphs of Example~\ref{Examp4} are weakly connected for any pair of $k$ and $m$. $LLSM$ weights can be obtained as the solution of the following system of linear equations (note that alternatives of $C_i$ have the same weight $y_{C_i}$ since the $LLSM$ weight vector is unique and alternatives of $C_i$ are symmetric):
\begin{eqnarray}
(k + 1) y_1 - y_2 - k y_{C_m} & = & k + 1; \quad \text{(for the $1$st alternative)} \label{eq1} \\
(k + 1) y_2 - y_1 - k y_{C_1} & = & k - 1; \quad \text{(for the $2$nd alternative)} \label{eq2} \\
(k + 1) y_{C_1} - y_2 - k y_{C_2} & = & k - 1; \quad \text{(for alternatives in $C_1$)} \label{eq3} \\
2k y_{C_i} - k y_{C_{i-1}} - k y_{C_{i+1}} & = & 0; \quad \text{(for alternatives in $C_i$, $i = 2,3, \dots m-1$)} \label{eq4} \\
(k + 1) y_{C_m} - y_1 - k y_{C_{m-1}} & = & -(k + 1). \quad \text{(for alternatives in $C_m$)} \label{eq5}
\end{eqnarray}
For instance, the $2$nd alternative has $k+1$ known comparisons: with the first, and with all $k$ alternatives of $C_1$,
and the $2$nd is preferred in the latter $k$ of these.
The derivation can be found in \citet{BozokiFulopRonyai2010}.

Subtract \eqref{eq2} from \eqref{eq1} in order to get
\begin{eqnarray}
(k + 2) \left( y_1 - y_2 \right) - k \left( y_{C_m} - y_{C_1} \right) & = & 2. \label{eq6} 
\end{eqnarray}
The difference of equations \eqref{eq3} and \eqref{eq5} gives,
\begin{eqnarray}
\left( y_{C_1} - y_{C_m} \right) - \left( y_2 - y_1 \right) - k \left( y_{C_2} - y_{C_1} \right) - k \left( y_{C_m} - y_{C_{m-1}} \right) & = & 2k. \label{eq7}
\end{eqnarray}
It follows from equations \eqref{eq4} that
\begin{equation}
y_{C_2} - y_{C_1} = y_{C_3} - y_{C_2} = \cdots = y_{C_m} - y_{C_{m-1}} = \frac{y_{C_m} - y_{C_1}}{m-1}. \label{eq8}
\end{equation}
Equations \eqref{eq7} and \eqref{eq8} lead to
\begin{eqnarray*}
\left( y_{C_1} - y_{C_m} \right) - \left( y_2 - y_1 \right) + \frac{2k}{m-1} \left( y_{C_1} - y_{C_m} \right) & = & 2k,
\end{eqnarray*}
which results in
\begin{equation}
y_{C_1} - y_{C_m} = \frac{m-1}{2k + m -1} \left[ 2k + \left( y_2 - y_1 \right) \right]. \label{eq9}
\end{equation}
Substituting \eqref{eq9} into \eqref{eq6} gives
\begin{equation*}
(k + 2) \left( y_1 - y_2 \right) = 2 - k \left( y_{C_1} - y_{C_m} \right) = 2- \frac{k(m-1)}{2k + m -1} \left[ 2k + \left( y_2 - y_1 \right) \right].
\end{equation*}
After some calculation we infer
\begin{eqnarray*}
(k + 2)(2k + m -1) \left( y_1 - y_2 \right) & = & 4k + 2m - 2 - 2k^2(m-1) - k(m-1) \left( y_2 - y_1 \right); \\
\left( 2k^2 + 4k + 2m - 2 \right) \left( y_1 - y_2 \right) & = & -2k^2(m-1) + 4k + 2m - 2.
\end{eqnarray*}
It means that
\begin{equation*}
y_1 - y_2 = \frac{-k^2(m-1) + 2k + m - 1}{k^2 + 2k + m - 1},
\end{equation*}
so $\lim_{m \to \infty} \left( y_1 - y_2 \right) = -\infty$ for any $k \geq 2$.
Hence $w_2^{LLSM} - w_1^{LLSM}$ can be arbitrarily large independent of $b$.\footnote{~It is trivial if $b$ can vary as $\log b$ may also be arbitrarily large.}
\end{proof}

\begin{remark} \label{Rem3}
In Example~\ref{Examp4}, $w_1^{LLSM} - w_2^{LLSM} > 0$ if $k = 2$ and $m = 2$ (so there are $6$ alternatives), but $w_1^{LLSM} - w_2^{LLSM} < 0$ if $k = 3$ and $m = 2$ or $k = 2$ and $m = 3$ (so there are $8$ alternatives with $13$ and $16$ known comparisons, respectively). This family of directed acyclic graphs does not give an example with $7$ alternatives as Example~\ref{Examp2} does.
\end{remark}

\section{Linear order preservation and the Eigenvector \\ Method} \label{Sec4}

In this section we examine whether $EM$ satisfies the property $LOP$ or not.

\begin{proposition} \label{Prop2}
Let $\mathbf{A} = \left[ a_{ij} \right]$ be an ordinal pairwise comparison matrix.
Vector $\mathbf{w}^{EM}(\mathbf{A})$ is unique if and only if the directed graph associated to $\mathbf{A}$ is weakly connected.
\end{proposition}

\begin{proof}
See \citet[Theorem 2]{BozokiFulopRonyai2010}.
\end{proof}

\begin{theorem} \label{Theo3}
$EM$ may violate $LOP$.
\end{theorem}

\begin{proof}
Consider the directed acyclic graph on Figure~\ref{Fig3}. An associated ordinal pairwise comparison matrix $\mathbf{A}$ is as follows:
\[
\mathbf{A}=
\left(
\begin{array}{cccc cccc}
    1     & 3     & \ast  & \ast  & \ast  & \ast  & 3     & 3     \\
     1/3  & 1     & 3     & 3     & \ast  & \ast  & \ast  & \ast \\
    \ast  &  1/3  & 1     & \ast  & 3     & 3     & \ast  & \ast \\
    \ast  &  1/3  & \ast  & 1     & 3     & 3     & \ast  & \ast \\
    \ast  & \ast  &  1/3  &  1/3  & 1     & \ast  & 3     & 3     \\
    \ast  & \ast  &  1/3  &  1/3  & \ast  & 1     & 3     & 3     \\
     1/3  & \ast  & \ast  & \ast  &  1/3  &  1/3  & 1     & \ast \\
     1/3  & \ast  & \ast  & \ast  &  1/3  &  1/3  & \ast  & 1     \\
\end{array}
\right).
\]

Property $LOP$ is satisfied if $w_1(\mathbf{A}) \geq w_2(\mathbf{A})$, $w_1(\mathbf{A}) \geq w_7(\mathbf{A})$, $w_1(\mathbf{A}) \geq w_8(\mathbf{A})$, $w_2(\mathbf{A}) \geq w_3(\mathbf{A})$, $w_2(\mathbf{A}) \geq w_4(\mathbf{A})$ as well as $w_i(\mathbf{A}) \geq w_j(\mathbf{A})$ for all $i = 3,4$ and $j = 5,6$; $i = 5,6$ and $j = 7,8$.

However, $EM$ results in
\[
\mathbf{w}^{EM}(\mathbf{A}) = \left[
\begin{array}{cccc cccc}
    0.2404 & 0.2442 & 0.1481 & 0.1481 & 0.0729 & 0.0729 & 0.0367 & 0.0367 \\
\end{array}
\right]^\top,
\]
that is, $w_1^{EM}(\mathbf{A}) < w_2^{EM}(\mathbf{A})$, in contradiction with preservation of linear order.
\end{proof}

A parallel of Corollary~\ref{Col1} is not true in the case of $EM$, it may give a different ranking for a certain ordinal pairwise comparison matrix corresponding to the same directed acyclic graph.

\begin{proposition} \label{Prop3}
Let $\mathbf{A} = \left[ a_{ij} \right]$ be an ordinal pairwise comparison matrix representing a directed acyclic graph.
The ranking of the alternatives according to $EM$ depends on the value of $b>1$.
%, that is, $w^{EM}_i (\mathbf{A}) \geq w^{EM}_j (\mathbf{A})$ and $w^{EM}_i (\mathbf{A}^{(h)}) \geq w^{EM}_j (\mathbf{A}^{(h)})$ may occur for some $i,j = 1,2, \dots ,n$ and $h>0$, where $\mathbf{A}^{(h)} = \left[ a_{ij}^h \right]$.
\end{proposition}

\begin{proof}
Consider the directed acyclic graph on Figure~\ref{Fig3}. Besides $\mathbf{A}$, another representation by ordinal pairwise comparison matrix $\mathbf{A}'$ is as follows:
\[
\mathbf{A}'=
\left(
\begin{array}{cccc cccc}
    1     & 4     & \ast  & \ast  & \ast  & \ast  & 4     & 4     \\
     1/4  & 1     & 4     & 4     & \ast  & \ast  & \ast  & \ast \\
    \ast  &  1/4  & 1     & \ast  & 4     & 4     & \ast  & \ast \\
    \ast  &  1/4  & \ast  & 1     & 4     & 4     & \ast  & \ast \\
    \ast  & \ast  &  1/4  &  1/4  & 1     & \ast  & 4     & 4     \\
    \ast  & \ast  &  1/4  &  1/4  & \ast  & 1     & 4     & 4     \\
     1/4  & \ast  & \ast  & \ast  &  1/4  &  1/4  & 1     & \ast \\
     1/4  & \ast  & \ast  & \ast  &  1/4  &  1/4  & \ast  & 1     \\
\end{array}
\right).
\]

$EM$ gives
\[
\mathbf{w}^{EM}(\mathbf{A}') = \left[
\begin{array}{cccc cccc}
    0.2828 & 0.2656 & 0.1404 & 0.1404 & 0.0594 & 0.0594 & 0.0260 & 0.0260 \\
\end{array}
\right]^\top,
\]
thus $w_1^{EM}(\mathbf{A}) < w_2^{EM}(\mathbf{A})$ but $w_1^{EM}(\mathbf{A}') > w_2^{EM}(\mathbf{A}')$, which verifies Proposition~\ref{Prop3}.
\end{proof}

Note that Proposition~\ref{Prop3} does not necessarily mean the violation of linear order preservation if the $a_{ij}$ is missing. For instance, the relative ranking of alternatives $6$ and $7$ can be arbitrary in Example~\ref{Examp2}.

\begin{remark} \label{Rem4}
There exists an example with $6$ alternatives and $5$ known comparisons (which is minimal provided weak connectedness) demonstrating Proposition~\ref{Prop3}. It is the smallest in the number of alternatives.
\end{remark}

The sensitivity of the ranking of the alternatives by $EM$ to the choice of $b>1$ was observed by \citet{GenestLapointeDrury1993} for certain complete pairwise comparison matrices.

\section{Conclusion} \label{Sec5}

Logarithmic Least Squares Method seems to give a counter-intuitive ranking of the alternatives for some incomplete pairwise comparison matrices representing preferences described by a directed acyclic graph. The ranking according to the Eigenvector Method  may also contradict to the natural ranking order, while it depends on the correspondence chosen for these preferences, too.

Our results open at least three topics for future research:
\begin{enumerate}
\item
How can one characterize the set of ordinal pairwise comparison matrices with a linear order of the alternatives for which $LLSM$ obeys $LOP$?\footnote{~One may realize that all counterexamples have a certain structure, visible on the family of directed acyclic graphs according to Example~\ref{Examp4} the first alternative is preferred to the second and some with a low rank, while the second is preferred to some others with a high rank.}
\item
When does an unambiguous ranking of the alternatives according to $EM$ exist (on the class of preferences given by a directed acyclic graph)?
\item
Which weighting methods perform well with respect to the condition $LOP$?
\end{enumerate}

%\bibliographystyle{apalike}
%\bibliography{All_references}

\end{document}